\documentclass[11pt]{amsart}

\newtheorem{theorem}{Theorem}[section]

\newtheorem{corollary}[theorem]{Corollary}

\newtheorem{lemma}[theorem]{Lemma}

\newtheorem{problem}[theorem]{Problem}
\newtheorem{proposition}[theorem]{Proposition}

\newtheorem{example}[theorem]{Example}

\theoremstyle{definition}
\newtheorem{remark}[theorem]{Remark}

\usepackage[colorlinks, bookmarks=true]{hyperref}
\usepackage{color,graphicx,shortvrb}

\usepackage{enumerate}
\usepackage{amsmath}
\usepackage{amssymb}

\usepackage[latin 1]{inputenc}

\def\J#1#2#3{ \left\{ #1,#2,#3 \right\} }

\def\11{\textbf{$1$}}

\begin{document}

\numberwithin{equation}{section}

\title[]{A note on 2-local representations of C$^*$-algebras}

\author[A.M. Peralta]{Antonio M. Peralta}
\address{Departamento de An{\'a}lisis Matem{\'a}tico, Universidad de Granada,\\
Facultad de Ciencias 18071, Granada, Spain}
\email{aperalta@ugr.es}

\thanks{Author partially supported by the Spanish Ministry of Science and Innovation,
D.G.I. project no. MTM2011-23843, and Junta de Andaluc\'{\i}a grant FQM375.}



\date{}
\maketitle

\begin{abstract} We survey the results on linear local and 2-local homomorphisms and zero products preserving operators between C$^*$-algebras, and we incorporate some new precise observations and results to prove that every bounded linear 2-local homomorphism between C$^*$-algebras is a homomorphism. Consequently, every linear 2-local $^*$-homomorphism between C$^*$-algebras is a $^*$-homomorphism.
\end{abstract}

\maketitle
\thispagestyle{empty}

\section{Introduction}\label{sec:intro}

Most of the authors agree in acknowledging the papers of R.V. Kadison \cite{Kad90} and D.R. Larson and A.R. Sourour \cite{LarSou} as the pioneering contributions to the theory of local derivations and local automorphisms on Banach algebras, respectively. We recall that a linear mapping $T$ from a Banach algebra $A$ into a Banach algebra $B$ is said to be a \emph{local homomorphism} if for every $a$ in $A$ there exists a homomorphism  $\Phi_{a} : A \to B$, depending on $a$, satisfying $T(a) = \Phi_a (a)$. When $A$ and $B$ are C$^*$-algebras, and for each $a$ in $A$ there exists a  $^*$-homomorphism $\Phi_{a} : A \to B$, depending on $a$, satisfying $T(a) = \Phi_a (a)$, the mapping $T$ is called a \emph{local $^*$-homomorphism}. \emph{Local automorphisms}, \emph{local $^*$-automorphisms}, and \emph{local derivations} are similarly defined.\smallskip

R.V. Kadison proved in \cite{Kad90} that every bounded local derivation on a von Neumann algebra (i.e. a C$^*$-algebra which is also a dual Banach space) is a derivation. After Kadison's contribution, a multitude of researchers explored the same problem for general C$^*$-algebras (see, for example, \cite{AlBreExVill, Bre92, HadLi04, LiPan, Shu}, and \cite{ZhanPanYang}). The definitive answer is due to B.E. Johnson  \cite{John01}, who proved that every local derivation from a C$^*$-algebra $A$ into a Banach $A$-bimodule is a derivation, even if not assumed a priori to be so. Much more recently, local triple derivations on C$^*$-algebras and JB$^*$-triples have been studied in \cite{Mack, BurFerGarPe2012, BurFerPe2013} and \cite{FerMolPe}.\smallskip

The knowledge about local homomorphisms and local $^*$-homomorphisms between C$^*$-algebras is less conclusive. D.D. Larson and A.R. Sourour proved in \cite{LarSou} that for an infinite dimensional Banach space $X$, every surjective local automorphism $T$ on the Banach algebra $B(X),$  of all bounded linear operators on $X$, is an automorphism. When $X$ is a separable Hilbert space M. Bre\v{s}ar and P. \v{S}emrl showed that the hypothesis concerning the surjectivity is superfluous (cf. \cite{BreSemrl95}). A related result was established by C. Batty and L. Molnar in \cite{BattMol}, where they proved that for a properly infinite von Neumann algebra $\mathcal{M}$, the group, Aut$(\mathcal{M})$, of all $^*$-automorphisms on $\mathcal{M}$ is reflexive, i.e. if a linear mappings $T:\mathcal{M}\to \mathcal{M}$ satisfies that for every $a\in \mathcal{M}$, $T(a)$ belongs to the strong-closure of the set $\{\Phi (a): \phi \in \hbox{Aut}(\mathcal{M})\}$, then $T$ lies in Aut$(\mathcal{M})$. Furthermore, for each Hilbert space $H$ of dimension $n\geq 3$, the group Aut$(B(H))$ is reflexive. In \cite[\S 2]{Pop}, F. Pop provides an example of a local homomorphism from $M_2 (\mathbb{C})$ into $M_4 (\mathbb{C})$ which fails to be multiplicative (cf. Example \ref{example Pop revisited}).\smallskip

In 1997, P. \v{S}emrl \cite{Semrl97} introduces 2-local derivations and 2-local automorphisms in the following sense: Let $A$ be a Banach algebra, a mapping $T : A \to A$ is a \emph{2-local automorphism} if for every $a, b \in  A$ there is an automorphism $T_{a,b} : A \to  A,$
depending on $a$ and $b$, such that $T_{a,b} (a) = T(a)$ and $T_{a,b}(b) = T(b)$ (no linearity, surjectivity or continuity of $T$ is assumed). In the just quoted paper, \v{S}emrl proves that for every infinite-dimensional separable Hilbert space $H$, every 2-local automorphism $T: B(H) \to B(H)$ is an automorphism. In \cite{KimKim05}, S.O. Kim and J.S. Kim show that every surjective 2-local $^*$-automorphism on a prime C$^*$-algebra or on a C$^*$-algebra such that the identity element is properly infinite is a $^*$-automorphism (see \cite{CristLocalAutomorph,Fos2012,Mol2002,Mol2003, Mol2007, MolSemrl} and \cite{KimKim04} for other related results).\smallskip

A closer look at \v{S}emrl's paper \cite{Semrl97} shows that the connections with the Gleason-Kahane-\.{Z}elazko theorem (cf. \cite{Gle,KaZe}) didn't go unnoticed to him. Borrowing a paragraph from \cite[Introduction]{Semrl97}, we notice that Gleason-Kahane-\.{Z}elazko theorem can be reformulated in the following sense: every unital linear local homomorphism from a unital complex Banach algebra $A$ into $\mathbb{C}$ is multiplicative (cf. \cite{Badea1993}). S. Kowalski and Z. Slodkowski \cite{KoSlod} established a 2-local version of the Gleason-Kahane-\.{Z}elazko theorem, showing that every 2-local homomorphism $T:A \to \mathbb{C}$ is linear and multiplicative.\smallskip

In order to keep coherence with the terminology employed by P. \v{S}emrl, a mapping $T$ between C$^*$-algebras $A$ and $B$ is called a \emph{2-local homomorphism} (respectively, \emph{2-local $^*$-homomorphism}) if for every $a,b\in A$ there exists a bounded homomorphism (respectively, a \emph{$^*$-homomorphism}) $\Phi_{a,b}: A\to B$, depending on $a$ and $b$, such that $\Phi_{a,b} (a) = T(a)$ and $\Phi_{a,b}(b) = T(b)$. \emph{2-local Jordan homomorphisms}, \emph{2-local Jordan $^*$-homomorphisms} and \emph{2-local automorphisms} are defined in a similar fashion. We recall that a linear mapping $\Phi: A\to B$ is said to be a Jordan homomorphism whenever $\Phi (a^2) = \Phi (a)^2$ (equivalently, $\Phi$ preserves the Jordan products of the form $a\circ b := \frac12 (a b + ba)$).\smallskip

In 2004, new studies on 2-local linear maps between C$^*$-algebras were developed by D. Hadwin and J. Li \cite{HadLi04} and F. Pop \cite{Pop}, though these papers seem to be mutually disconnected at the publication moment. Hadwin and Li prove that every bounded linear and unital 2-local homomorphism (respectively, 2-local $^*$-homomorphism) from a unital C$^*$-algebra of real rank zero into itself is a homomorphism (respectively, a $^*$-homomorphism) \cite[Theorem 3.6]{HadLi04}. As a consequence, every linear and surjective 2-local $^*$-automorphism on a unital C$^*$-algebra of real rank zero is a $^*$-automorphism (cf. \cite[Theorem 3.7]{HadLi04}). The main contribution of F. Pop in \cite{Pop} establishes that every bounded linear 2-local homomorphism (respectively, 2-local $^*$-homomorphism) from a von Neumann algebra into another C$^*$-algebra is a homomorphism (respectively, a $^*$-homomorphism) \cite[Corollary 3.6]{Pop}.\smallskip

In 2006, J.-H. Liu and N.-C. Wong made their own contribution to the study of not necessarily continuous nor linear 2-local homomorphisms between standard operator algebras on locally convex spaces \cite{LiuWong06}. We recall that a standard operator algebra $\mathcal{A}$ on a locally convex space $X$, is a subalgebra of $B(X)$ containing the algebra $\mathcal{F} (X)$ of all continuous finite rank operators on $X$.  Liu and Wong prove, without assuming linearity, surjectivity or continuity, that every 2-local automorphism of $\mathcal{F} (X)$ is an algebra homomorphism. In
case $X$ is a Frechet space with a Schauder basis and $\mathcal{A}$ contains all locally compact operators, it can be concluded that every 2-local automorphism on $\mathcal{A}$ is an automorphism. Furthermore, a 2-local automorphism $\Theta$ of a standard operator algebra $\mathcal{A}$ on a locally convex space $X$ is an algebra homomorphism provided that the range of $\Theta$ contains $\mathcal{F} (X)$, or $\Theta$ is continuous in the weak operator topology (cf. \cite{LiuWong06}). In the just quoted paper, the authors study the question of when a 2-local automorphism of a C$^*$-algebra is an automorphism, showing that every linear 2-local automorphism $T$ of a C$^*$-algebra whose range is a C$^*$-algebra is an algebra homomorphism.\smallskip

It seems natural to ask whether the above results of Hadwin-Li and Pop remain true for general C$^*$-algebras. This paper, which has an almost expository aim, combined with new research results, we give a positive answer to this question, showing that every bounded linear 2-local homomorphism between C$^*$-algebras is a homomorphism, and consequently, every linear 2-local $^*$-homomorphism between C$^*$-algebras is a $^*$-homomorphism (Theorem \ref{t 2-local $^*$-homomorphisms}). In particular, according to the terminology in \cite{Pop}, every 2-local ($^*$-)representation of a C$^*$-algebra is a ($^*$-)representation (Corollary \ref{c 2-local linear representations}). In Example \ref{example unitarily equivalent matrices} we present a linear 2-local $^*$-automorphism on $M_2 (\mathbb{C})$ which is not multiplicative.  We survey the connections between this problem and the theory of linear zero products preservers developed by J. Alaminos, M. Bresar, J. Extremera and A. Villena in \cite{AlBreExVill}. The novelties in this paper include an independent proof which is not based on the result in \cite{AlBreExVill} together with the precise observations to provide a definitive answer to the whole line of problems on linear preservers on C$^*$-algebras presented above.
Here we make use of techniques developed in the setting of JB$^*$-triples, the use of compact-G$_{\delta}$ projections in the bidual of a C$^*$-algebra, and the study of the connections between (linear) 2-local homomorphisms and zero product preserving mappings. Although the results presented here could have been obtained by combining some of the results that we shall review later, the equivalence between bounded linear 2-local homomorphisms and bounded homomorphisms between C$^*$-algebras has not been explicitly stated before.

\section{Techniques of Jordan algebras and JB$^*$-triples}

Every C$^*$-algebra $A$ admits a Jordan product defined by $a\circ b = \frac12 (a b +b a)$. The Jordan product is commutative but not necessarily associative. Let $B$ be another C$^*$-algebra. A linear map $T: A\to B$ is said to be a \emph{Jordan homomorphism} whenever it preserves Jordan products, or equivalently, when $T(a^2) = T(a)^2$, for every $a$. A \emph{Jordan $^*$-homomorphism} is a {Jordan homomorphism} which maps self-adjoint elements into self-adjoint elements. For each element $a$ in $A$, the symbol $U_a$ will denote the linear map $U_a : A \to A$ defined by $U_a (x) := a x a$. Since, for every $a,x\in A$ we have $U_a (x) = 2 (a\circ x) \circ a - a^2 \circ x$, every Jordan homomorphism $T: A\to B$ satisfies $T (U_a (x)) = U_{T(a)} (T(x))$.\smallskip

Let $T: A \to B$ be a Jordan homomorphism between C$^*$-algebras. Since $A^{**}$ and $B^{**}$ are von Neumann algebras, $T^{**} : A^{**} \to B^{**}$ is weak$^*$ continuous, and the product of every von Neumann algebra is separately weak$^*$ continuous (cf. \cite[Theorem 1.7.8]{Sak}), we deduce, via Goldstine's theorem, that $T^{**} : A^{**}\to B^{**}$ is a Jordan homomorphism. Since the involution of a von Neumann algebra is weak$^*$ continuous (cf. \cite[Theorem 1.7.8]{Sak}), $T^{**}$ is a Jordan $^*$-homomorphism whenever $T$ is a Jordan $^*$-homomorphism.\smallskip

There is some benefit in considering a C$^*$-algebra as an element in the wider class of JB$^*$-triples. A \emph{JB$^*$-triple} is a complex Banach space $E$ equipped with a triple product $\{\cdot,\cdot,\cdot\}:E\times E\times E\rightarrow E$ which is linear and symmetric in the outer variables, conjugate
linear in the middle variable and satisfies the following conditions: \begin{enumerate}[$(a)$]
\item (Jordan identity)
$$\{a,b,\{x,y,z\}\}=\{\{a,b,x\},y,z\}
-\{x,\{b,a,y\},z\}+\{x,y,\{a,b,z\}\},$$ for $a,b,x,y,z$ in $E$;
\item For each $a\in E$, the mapping $L(a,a):E\rightarrow E,$ $x\mapsto \{a,a,x\}$ is an hermitian
(linear) operator with non-negative spectrum; 
\item $\|\{x,x,x\}\|=\|x\|^3$ for all $x\in E$.
\end{enumerate}

Every C$^*$-algebra is a JB$^*$-triple via the triple product given by $$ \J xyz = \frac12 (x y^* z +z y^* x).$$

It was shown by Poincaré in the early 1900s, that the Riemann mapping theorem fails when the complex plane is replaced by a complex Banach space of higher dimension. Although, a complete holomorphic classification of bounded simply connected domains in arbitrary complex Banach spaces is unattainable, {bounded symmetric domains} in finite dimensions were studied and classified by E. Cartan \cite{Cartan35}. In the setting of complex Banach spaces of arbitrary dimension, W. Kaup proved, in \cite{Ka83}, that a complex Banach space is a JB$^*$-triple if, and only if, its open unit ball is a bounded symmetric domain, and every bounded symmetric domain in a complex Banach space is biholomorphically equivalent to the open unit ball of a
JB$^*$-triple; showing that the category of all bounded symmetric domains with base point is equivalent to the category of JB$^*$-triples. We refer to monographs \cite{Up} and \cite{Chu2012} for the basic theory of JB$^*$-triples and JB$^*$-algebras.\smallskip

Spectral resolutions of non-normal elements in a C$^*$-algebra is a completely hopeless goal. However, in every JB$^*$-triple $E$,
the JB$^*$-subtriple $E_a$ generated by a single element $a\in E$ is (isometrically) JB$^*$-isomorphic to $C_0 (L)$ for some locally compact Hausdorff space $L\subseteq (0,\|a\|]$, such that $L\cup \{0\}$ is compact. It is also known that there exists a JB$^*$-triple isomorphism $\Psi_a: E_a\to C_{0}(L),$ satisfying $\Psi (a) (t) = t$ $(t\in L)$ (compare \cite[Lemma 1.14]{Ka83}). In particular, there exists a unique element $a^{[\frac13]}\in E_a$ such that $\{a^{[\frac13]}, a^{[\frac13]}, a^{[\frac13]}\} = a.$ When $E=A$ is a C$^*$-algebra, $$a^{[\frac13]} (a^{[\frac13]})^* a^{[\frac13]}= \{a^{[\frac13]}, a^{[\frac13]}, a^{[\frac13]}\} = a.$$ In order to simplify notation, for each element $a$ in a JB$^*$-triple $E$ we write $a^{[1]} =
a$ and $a^{[2 n +1]} := \J a{a^{[2n-1]}}a$ $(\forall n\in \mathbb{N})$. It is known that JB$^*$-triples
are power associative, that is, $\J{a^{[2 k-1]}}{a^{[2 l-1]}}{a^{[2 m-1]}}=a^{[2(k+l+m)-3]},$ for every $k,l,m\in\mathbb{N}$ (cf. \cite[Lemma 1.2.10]{Chu2012}).

\section{Local and 2-local representations of C$^*$-algebras}

Let $A$ and $B$ be C$^*$-algebras. Clearly, every local $^*$-homomorphism $T:A\to B$ is automatically continuous and contractive. Indeed, since for each $a\in A$, there exists a $^*$-homomorphism $\Phi_{a} : A \to B$ satisfying $T(a) = \Phi_a (a)$, we have $\|T(a) \| = \| \Phi_a (a) \| \leq \|a\|$. Concerning (local) homomorphisms, many basic questions are still open, like automatic continuity of homomorphisms between C$^*$-algebras (\cite[Question 5.4.D]{Dales00} or \cite[Question 1]{Vill}).\smallskip

The next result summarizes some clear facts about local homomorphisms.

\begin{lemma}\label{l composition with homo} Let $A,B$ and $C$ denote C$^*$-algebras, $T: A\to B$ a local homomorphism {\rm (}respectively, a local $^*$-homomorphism{\rm )} and $\Phi: B \to C$ a homomorphism {\rm (}respectively, a $^*$-homomorphism{\rm )}, then $\Phi T$ is a local homomorphism {\rm (}respectively, a local $^*$-homomorphism{\rm )}. Every local $^*$-homomorphism between C$^*$-algebras is positive.$\hfill\Box$
\end{lemma}

By Gelfand theory the set of extreme points in the positive part of the unit ball in the dual, $\mathcal{B}^*$, of a commutative C$^*$-algebra $\mathcal{B}$ is precisely the set $X$ of non-zero multiplicative functionals on $\mathcal{B}$, and thus the identification $\mathcal{B} = C_0(X)$ was established. Therefore, non-zero homomorphisms from $C_0(L)$ into $\mathbb{C}$ identify with those functionals $\delta_{s} : C_0(L) \to \mathbb{C}$, $\delta_s (f) =f(s),$ where $s$ runs in $L.$\smallskip

The refinement of the Gleason-Kahane-\.{Z}elazko theorem established by \.{Z}elazko in \cite{Ze68} asserts that for every complex Banach algebra $\mathcal{B}$ (not necessarily unital nor commutative), every linear selection from the spectrum $\varphi : \mathcal{B} \to \mathbb{C}$ (i.e. $\varphi (a) \in \sigma (a)$, for every $a\in \mathcal{B}$) is multiplicative. Another interesting result, implicitly established by J.P. Kahane and W. \.{Z}elazko in \cite[Theorem 3]{KaZe}, will be applied in the next proposition.\smallskip

\begin{proposition}\label{p local hom between C(K)} Let $L_1$ and $L_2$ be locally compact Hausdorff spaces and let $T: C_0(L_1) \to C_0(L_2)$ be a local homomorphism. Then, for each $s\in L_2$, the mapping $\delta_{s} T: C_0(L_1)\to \mathbb{C}$ is a $^*$-homomorphism. In particular, $T$ is a $^*$-homomorphism.
\end{proposition}

\begin{proof} Let us assume that $\delta_{s} T\neq 0$. We shall prove that $\delta_{s} T = \delta_{t}$ for a unique $t\in L_1$. Since $T$ is a local homomorphism, $\delta_{s} T$ is a local homomorphism. So, given $f\in C_0(L_1)$ there exists a homomorphism $\Phi_{f} : C_0(L_1) \to \mathbb{C}$, and hence an element $t_f\in L_1$, satisfying $\delta_{s} T (f) = \Phi_{f} (f) = f(t_f).$ Now, Theorem 3 in \cite{KaZe} proves that $\delta_{s} T $ is a multiplicative functional.
\end{proof}

\begin{corollary}\label{c abelian codomain} Let $T: A \to B$ be a local homomorphism between C$^*$-algebras, where $B$ is commutative.
Then $T$ is a Jordan $^*$-homomorphism and, consequently, is continuous.
\end{corollary}

\begin{proof} 
Let $a$ be a self adjoint element in $A$. Considering the C$^*$-subalgebra, $\mathcal{C},$ generated by $a$, the mapping $T|_{\mathcal{C}}: \mathcal{C}\to B$ is a local homomorphism between commutative C$^*$-algebras. Proposition \ref{p local hom between C(K)} assures that $T|_{\mathcal{C}}$ is a $^*$-homomorphism. Therefore $T(a^2) = T(a)^2$ for every $a\in A_{sa}$, and hence $T$ is a Jordan $^*$-homomorphism.
\end{proof}

When $C(K)$ is replaced with the real C$^*$-algebra $C(K,\mathbb{R})$, of all real-valued continuous functions on $K$, the corresponding versions of the above results are not, in general, true. For example, the linear operator $T: C([0,1],\mathbb{R}) \to \mathbb{R}$, $\displaystyle T(a) :=\int_{0}^{1} a(t) dt$ is not multiplicative. However, the mean-value theorem implies that $T$ is a local homomorphism. \smallskip

The existence of bounded linear operators between C$^*$-algebras which are local homomorphisms and fail to be multiplicative (cf. \cite[example in \S 2]{Pop}) led F. Pop to focus his attention on 2-local homomorphisms (called 2-local representations by Pop). The just-mentioned counter-example, provided by Pop, is not multiplicative but it is a Jordan homomorphism (see Example \ref{example Pop revisited}). The latter property is actually satisfied by every linear 2-local homomorphism between C$^*$-algebras (cf. \cite[Lemma 2.1]{LiuWong06}).\smallskip

\begin{proposition}\label{t 2-local *-homo between C*-algebras} Every linear 2-local homomorphism between C$^*$-algebras is a Jordan homomorphism. Every linear 2-local $^*$-homomorphism between C$^*$-algebras is a Jordan $^*$-homomorphism.
\end{proposition}

\begin{proof} When $T$ is a linear 2-local homomorphism, for each $a\in A$, there exists a homomorphism $\Phi_{a,a^2}: A\to B$ such that $T(a) = \Phi_{a,a^2} (a)$ and $T(a^2) = \Phi_{a,a^2} (a^2)$. Then, $T(a)^2 = \Phi_{a,a^2} (a)^2 =  \Phi_{a,a^2} (a^2) = T(a^2),$  confirming that $T$ is a Jordan homomorphism.
\end{proof}

In the setting of von Neumann algebras, the hypothesis in Proposition \ref{t 2-local *-homo between C*-algebras} can be relaxed. Indeed, in \cite[Proposition 1.4]{Pop}, F. Pop establishes that every bounded linear local homomorphism from a commutative von Neumann algebra into $B(H)$ is multiplicative, and hence a representation. This result applies to get:

\begin{corollary}\label{c Pop Jordan} Let $T: M \to B$ be a bounded linear local homomorphism from a von Neumann algebra into a C$^*$-algebra. Then $T$ is a Jordan homomorphism. Consequently, every linear local $^*$-homomorphism from a von Neumann algebra into a C$^*$-algebra is a Jordan $^*$-homomorphism.
\end{corollary}

\begin{proof} Let $T: M \to B$ be a bounded linear local homomorphism. Making use of the representation theory and Lemma \ref{l composition with homo}, we can assume that $B=B(H)$ for a suitable complex Hilbert space $H$. Let $a$ be a self-adjoint element in $M$, and let $\mathcal{C}$ denote the von Neumann subalgebra of $M$ generated by $a$ and $1$. Clearly, $T|_{\mathcal{C}} : \mathcal{C}\to B(H)$ is a bounded linear local homomorphism. By  \cite[Proposition 1.4]{Pop}, $T|_{\mathcal{C}}$ is multiplicative. Therefore, $T(a^2) = T(a)^2,$ for every $a\in M_{sa}$. This implies that $T(a\circ b) = T(a) \circ T(b)$ for every $a,b\in M_{sa}$ and hence $T((a+ib)^2) = T (a^2 - b^2 + 2 i a\circ b) =  T (a)^2 - T(b)^2 + 2 i T(a)\circ T(b) = T(a+ib)^2$, for every $a,b\in M_{sa}$, which proves the statement.
\end{proof}

It seems natural to ask whether the above mentioned results of Hadwin-Li and Pop hold when the domain is a general C$^*$-algebra. We shall see that the answer is intrinsically related to zero-products preserving operators between C$^*$-algebras.\smallskip

Let $A$ and $B$ be C$^*$-algebras. A mapping $f: A\to B$ is said to be \emph{orthogonality preserving} on a subset $U\subseteq A$ when the implication $$a\perp b  \Rightarrow f(a)\perp f(b),$$ holds for every $a,b\in U$. We recall that elements $a, b$ in $A$ are said to be
\emph{orthogonal} (denoted by $a \perp b$) whenever $a b^* = b^* a=0$. When the implication $$a b =0 \Rightarrow f(a) f(b) =0$$ holds for every $a,b\in U$, we shall say that $f$ \emph{preserves zero products} or is \emph{zero products preserving} on $U$.
In the case $A=U$, we shall simply say that $f$ is \emph{orthogonality preserving} or that $f$ \emph{preserves zero products}, respectively. Every homomorphism between C$^*$-algebras preserves zero products and every $^*$-homomorphism is orthogonality preserving.\smallskip

\begin{lemma}\label{l local 2-homo are OP}(cf. \cite[Lemma 2.1]{LiuWong06}) Let $T: A\to B$ be a map between C$^*$-algebras. Suppose $T$ is a 2-local $^*$-homomorphism (respectively, a 2-local homomorphism), then $T$ is orthogonality preserving (respectively, zero products preserving).
\end{lemma}

\begin{proof} Given $a,b\in A$ with $a\perp b$, we take a $^*$-homomorphism $\Phi_{a,b} : A\to B$ satisfying $T(a) = \Phi_{a,b} (a)$ and $T(b) = \Phi_{a,b} (b)$. Clearly, $T(a) = \Phi_{a,b} (a)\perp  \Phi_{a,b} (b)= T(b)$. The other statement follows similarly.
\end{proof}

Orthogonality preserving bounded linear maps between C$^*$-algebras have been completely described in \cite[Theorem 17]{BurFerGarMarPe} (see \cite{BurFerGarPe} and \cite{GarPePuRa2013} for completeness). 
\smallskip

Let $A$ be a C$^*$-algebra. An element $x$ in the von Neumann algebra $A^{**}$ is a \emph{multiplier} for $A$ if $x A \subseteq A$ and $A x \subseteq A$. The symbol $M(A)$ will denote the set of all multiplier of $A$ in $A^{**}.$ It is known that $M(A)$ is a unital C$^*$-subalgebra of $A^{**}$. Multipliers are uninteresting if the algebra $A$ possesses a unit, because in such a case $M(A) =A$ (see \cite[\S 3.12]{Ped} for more details).\smallskip


Let $W$ be a von Neumann algebra. For each normal positive functional $\varphi\in W_*$ the mapping $W\times W \to \mathbb{C}$, $(x,y)_{\varphi} := \frac12 \varphi (x y^* + y^* x)$ defines a semi-positive sesquilinear form on $W$. The corresponding prehilbertian seminorm on $W$ is defined by $$\|x\|_{\varphi} := (x,x)_{\varphi}^{\frac12} = \left(\frac12 \varphi (x x^* + x^* x) \right)^{\frac12}.$$ The \emph{strong$^*$ topology} of $W$ (denoted by $s^*(W,W_*)$) is the locally convex topology on $W$ defined by all the seminorms $\|.\|_{\varphi}$, where $\varphi$ runs in the set of all positive functionals in $W_*$ (cf. \cite[Definition 1.8.7]{Sak}). It is known that the strong* topology of $W$ is compatible with the duality $(W,W_{*})$, that is a functional $\psi: W \to \mathbb{C}$ is $s^*(W,W_*)$ if and only if it is weak$^*$ continuous (see \cite[Corollary 1.8.10]{Sak}). It is also known, from the above fact together with the Grothedieck-Pisier-Haagerup inequality (cf. \cite{Haa}), that a linear map between von Neumann algebras is strong$^*$ continuous if and only if it is weak$^*$ continuous. We also recall that the product of every von Neumann algebra is jointly strong$^*$ continuous on bounded sets (see \cite[Proposition 1.8.12]{Sak}).\smallskip

The next result is a subtle variant of \cite[Lemma 2.2]{Wong07}. The proof applies techniques of JB$^*$-triples in a similar fashion to the arguments given in the proofs of \cite[Proposition 3.1]{BurFerGarPe}, \cite[Proposition 1.3]{GarPe12}, and \cite[Lemma 2.2]{Wong07}.

\begin{proposition}\label{p wong 2.2 refined} Let $T : A \to  B$ be a bounded linear map between C$^*$-algebras sending zero products in $A$ to zero products in $B$. Then the restricted map $T^{**}|_{M(A)}: M(A)\to B^{**}$ sends zero products in $M(A)$ to zero products in $B^{**}$.
\end{proposition}

\begin{proof} We fix $a,b\in M(A)$ with $a b=0$. For each natural, $n$, the odd triple power $a^{[3]} = a a^* a$, $a^{[2n +1]} = a (a^{[2n-1]})^* a$, satisfies that $a^{[2n-1]} b= 0$. Thus, we deduce that, $\alpha b= 0$, for every $\alpha$ in the JB$^*$-subtriple, $M(A)_a,$ of $M(A)$ generated by $a$. The same argument shows that \begin{equation}\label{eq 3 prop multipliers}\alpha \beta  = 0
\end{equation} for every $\alpha\in M(A)_{a}$ and $\beta\in M(A)_{b}$. Consequently, we have $a^{[\frac13]} b^{[\frac13]}=0.$\smallskip

Since $M(A)$ is a C$^*$-subalgebra of $A^{**}$, by Goldstine's Theorem, we can find bounded nets $(x_{\lambda})$ and $(y_{\mu})$ in $A$, converging in the weak$^*$ topology of $A^{**}$ to $a^{[\frac13]}$ and $b^{[\frac13]}$, respectively. The nets $\left(a^{[\frac13]} x_{\lambda}^* a^{[\frac13]}\right)$ and $ \left( b^{[\frac13]} y_{\mu}^* b^{[\frac13]}\right) $ lie in $A$, and $$\left(a^{[\frac13]} x_{\lambda}^* a^{[\frac13]}\right) \left( b^{[\frac13]} y_{\mu}^* b^{[\frac13]}\right) =0,$$ for every $\lambda$ and $\mu$.\smallskip

By hypothesis, $T$ is zero products preserving, and hence, \begin{equation}\label{p 1 equ 2} T\left(a^{[\frac13]} x_{\lambda}^* a^{[\frac13]}\right) T \left( b^{[\frac13]} y_{\mu}^* b^{[\frac13]}\right) =0,
\end{equation} for every $\lambda$ and $\mu$. Finally, taking weak$^*$-limits in $\lambda$ and $\mu$, the weak$^*$ continuity of $T^{**}$ and the separate weak$^*$-continuity of the product of $A^{**}$, together with \eqref{p 1 equ 2}, give $$0 = T^{**} \left(a^{[\frac13]} (a^{[\frac13]})^* a^{[\frac13]}\right) T^{**} \left( b^{[\frac13]} (b^{[\frac13]})^* b^{[\frac13]}\right) = T^{**} (a) T^{**} (b),$$ which completes the proof.
\end{proof}

Let $A$ be a C$^*$-algebra, a projection $p$ in $A^{**}$ is called \emph{compact-$G_{\delta}$} (relative to $A$) whenever there exists a positive, norm-one element $a$ in $A$ such that $p$ coincides with the weak$^*$-limit (in $A^{**}$) of the sequence $(a^{n})_n$. Following standard notation, we shall say that $p$ is a \emph{range projection} when there exists a positive, norm-one element  $a\in A$ for which $p$ is the weak$^*$-limit of the sequence $(a^{\frac1n})_n$.\smallskip

Our next result can be derived from \cite[Theorem 4.1]{AlBreExVill} (compare Remark \ref{remark Bresar}). To our knowledge, it has never been stated in the form presented here. We also include a new proof which is independent from the arguments in \cite{AlBreExVill}.

\begin{theorem}\label{t Jordan + 2-local between unital} Let $A$ and $B$ be C$^*$-algebras with $A$ unital. Let $J :A \to B$ be a bounded Jordan homomorphism preserving zero products. Then $J$ is a homomorphism.
\end{theorem}

\begin{proof} Since $J$ is a Jordan homomorphism, we deduce that $J(1)=e$ is an idempotent in $B$ and $$J(a) = J (U_1 (a) )= U_{J(1)} (J(a)) = U_{e} (J(a))= e J(a) e,$$ for every $a\in A$. Since $J^{**}: A^{**} \to B^{**}$ is a Jordan homomorphism too, we can actually assure that \begin{equation}\label{eq 00 t Jordan + 2 zpp} J^{**} (a) = e J^{**}(a) e = e J^{**}(a) = J^{**}(a) e,
\end{equation} for every $a\in A$.\smallskip

Since $J$ preserves zero products, given $a,b\in A$ with $a b=0$, we have $$J (ba) = J ( ab + ba ) = J(a) J(b) + J(b) J(a) = J(b) J(a), $$ and consequently
\begin{equation}\label{eq 0 t Jordan + 2 zpp} J (bza) =  J(bz) J(a) = J (b) J(za),
\end{equation} for every $a,b,z\in A$ with $ab=0$.\smallskip

Let us consider a compact-$G_{\delta}$ projection $p\in A^{**}$, that is, there exists a positive, norm-one element $a$ in $A$ such that $\displaystyle p=w^*-\lim_{n} a^{n}$. We can identify the C$^*$-subalgebra of $A$ generated by $1$ and $a$ with $C(K),$ where $K\subseteq [0,1]$, $1\in K$, and $a(t)= t$ in the corresponding identification. Let us define two sequences $(y_n)$ and $(z_n)$ in $C(K)$ given by
$$y_{n} (t):=\left\{%
\begin{array}{ll}
    1, & \hbox{if $t\in K\cap [0,1-\frac{1}{n}]$;} \\
    -2 n t +2 n -1, & \hbox{if $t\in K\cap [1-\frac1{n},1-\frac1{2n} ]$;} \\
    0, & \hbox{if $t\in K\cap [1-\frac1{2n},1 ]$,} \\
\end{array}%
\right.$$ and $$z_{n} (t):=\left\{%
\begin{array}{ll}
    0, & \hbox{if $t\in K\cap [0, 1-\frac1{3n}]$;} \\
    3 nt -3 n +1, & \hbox{if $t\in K\cap [1-\frac1{3n},1 ]$} \\
\end{array}%
\right. .$$ It is easy to check that $0\leq y_n, z_n$, $(y_n)$ is increasing, $(z_n)$ is decreasing, $y_n z_{m}= z_{m} y_n =0$ for every $n,m\in \mathbb{N}$, $m\geq n$, $w^*-\lim_{n} y_n = 1-p$, and $w^*-\lim_{n}  z_n =p$ in $A^{**}$.\smallskip

By hypothesis, for every $z$ in $A$, and every $n,m$ in $\mathbb{N}$ with $m\geq n$, we have $J(z z_m) J(y_n)=0$, and thus $$0 = w^*-\lim_{m\geq n} J(z z_m) J(y_n) = J^{**}(z p) J (y_n),\ \hbox{ for every } n\in \mathbb{N},$$ which implies that $$0 = w^*-\lim_{n} J^{**}(z p) J (y_n) =  J^{**}(z p) J (1-p),$$ and hence $$J^{**}(z p) e = J^{**}(z p) J(1) = J^{**}(z p) J (p).$$ It follows from \eqref{eq 00 t Jordan + 2 zpp} that \begin{equation}\label{eq 1 t Jordan + 2 zpp} J^{**} (zp) =  J^{**}(z p ) e = J^{**} (zp) J^{**}(p),
\end{equation} for every $z\in A^{**}$.\smallskip

Applying \eqref{eq 0 t Jordan + 2 zpp} we deduce that $$J (y_n) J(z z_m)= J (y_n z z_m),$$ and $$J (z_m) J(z y_n)= J (z_m z y_n),$$ for every $z\in A$, $n,m\in \mathbb{N}$ with $m\geq n$. Taking weak$^*$-limits in $m,n\to \infty$, we get \begin{equation}\label{eq 3.6.0} J^{**} (1-p) J^{**} (z p ) = J^{**} ((1-p)z p),
\end{equation} and \begin{equation}\label{eq 2 t Jordan + 2 zpp}J^{**} (p) J^{**} (z (1-p) ) = J^{**} (p z (1- p)),\end{equation} for every $z$ in $A$ or in $A^{**}.$ Combining \eqref{eq 3.6.0} with \eqref{eq 00 t Jordan + 2 zpp} we get $$J^{**} (z p ) - J^{**} (p) J^{**} (z p )= J^{**} (1) J^{**} (z p ) - J^{**} (p) J^{**} (z p ) = J^{**} (z p) - J^{**} (pz p),$$ and thus \begin{equation}\label{eq 3 t Jordan + 2 zpp} J^{**} (p) J^{**} (zp) =  J^{**}(p z p ),
\end{equation} for every $z$ in $A$ or in $A^{**}.$\smallskip

Now, combining \eqref{eq 2 t Jordan + 2 zpp} and \eqref{eq 3 t Jordan + 2 zpp}, we deduce that $$J^{**} (p) J^{**} (z) = J^{**} (p z) ,$$ for every $z\in A^{**}.$ We have therefore proved that \begin{equation}\label{eq 4 t Jordan + 2 zpp} J^{**} (p z ) =  J^{**}(p) J^{**} (z),
\end{equation} for every $z\in A^{**}$ and every compact-$G_{\delta}$ projection $p\in A^{**}$.\smallskip

Finally, take an arbitrary self adjoint element $b\in A$ and identify the C$^*$-subalgebra of $A$ generated by $1$ and $b$ with a $C(K)$-space for a suitable $K\subset [-\|b\|,\|b\|]$. The property proved in \eqref{eq 4 t Jordan + 2 zpp} shows that $$J^{**} (p z ) =  J^{**}(p) J^{**} (z)$$ for every projection $p$ of the form $p =\chi_{_{[\alpha,\beta]\cap K}}$ with $[\alpha,\beta]\subseteq [-\|b\|,\|b\|]$. Having in mind that projections $q\in C(K)^{**} \subseteq A^{**}$ of the form $q =\chi_{_{(\alpha,\beta)\cap K}},$ with $(\alpha,\beta)\subseteq K$ can be approximated in the strong$^*$ topology of $A^{**}$ by sequences of projections $\displaystyle (p_n) =\left( \chi_{_{[\alpha -\frac1n ,\beta+\frac1n]\cap K}}\right) $, we deduce that $J^{**} (q z ) =  J^{**}(q) J^{**} (z)$ for every such projection $q$ and every $z\in A^{**}$. It is well known that $b$ (regarded as an element in $C(K)\subseteq A$) can be approximated in norm by a finite linear combinations of mutually orthogonal projections of the form $\chi_{_{[\alpha,\beta]\cap K}}$ and $\chi_{_{(\alpha,\beta)\cap K}}$ with $[\alpha,\beta]\subseteq [-\|b\|,\|b\|]$ (i.e. steps functions). Therefore, $J (b) J(z) = J (b z)$, for every $z,b\in A$ with $b=b^*$ and, by linearity, $J$ is a homomorphism.
\end{proof}

We can now prove the main result concerning 2-local homomorphisms.

\begin{theorem}\label{t 2-local $^*$-homomorphisms} Every bounded linear 2-local homomorphism between C$^*$-algebras is a homomorphism. Every linear 2-local $^*$-homomorphism between C$^*$-algebras is a $^*$-homomorphism.
\end{theorem}

\begin{proof} Let $T: A\to B$ be a bounded 2-local homomorphism between C$^*$-algebras. Proposition \ref{t 2-local *-homo between C*-algebras} implies that $T$ is a Jordan homomorphism.\smallskip

By the 2-local property, we deduce, via Lemma \ref{l local 2-homo are OP}, that $T$ preserves zero products. Proposition \ref{p wong 2.2 refined} implies that $T^{**}|_{M(A)} : M(A) \to B^{**}$ preserves zero products.\smallskip

Finally, since  $T^{**}|_{M(A)} : M(A) \to B^{**}$ is a Jordan homomorphism which preserves zero products and $M(A)$ is unital, the above Theorem \ref{t Jordan + 2-local between unital} gives the desired statement.
\end{proof}

Clearly, Theorems 3.6 and 3.7 in \cite{HadLi04} are direct consequences of the above Theorem \ref{t 2-local $^*$-homomorphisms}.\smallskip

Accordingly to the notation in \cite{Pop}, given a C$^*$-algebra $A$ and a complex Hilbert space $H$, a bounded linear map $T : A \to  B(H)$ is called a \emph{2-local representation} of $A$ whenever it is a 2-local homomorphism. The next result generalizes \cite[Corollary 3.6]{Pop} to the general setting of C$^*$-algebras.

\begin{corollary}\label{c 2-local linear representations} Let $A$ be a C$^*$-algebra. Every 2-local representation of a $A$ is a representation.$\hfill\Box$
\end{corollary}

\begin{remark}\label{remark Bresar}{\rm It should be noted here that Theorem \ref{t Jordan + 2-local between unital} can be derived from \cite[Theorem 4.1]{AlBreExVill}.  Indeed, in the just commented result the authors prove that for every unital C$^*$-algebra $A$, every Banach algebra $B$, and every bounded linear operator $T: A\to B$ preserving zero products, then $$T(1)T(xy) = T(x)T(y),$$ for all $x,y$ in $A$. Therefore, if $J :A \to B$ is a bounded Jordan homomorphism preserving zero products, by the first part of the argument in the proof of Theorem \ref{t Jordan + 2-local between unital}, $J(1)=e$ is an idempotent in $B$ and $J(a) =e J(a) e= eJ(a) = J(a) e,$ for every $a\in A$, and hence $J(x y) = J(1) J(xy) = J(x)J(y),$ for all $x,y$ in $A$. That is, Theorem \ref{t Jordan + 2-local between unital} holds when $B$ is a Banach algebra. Proposition \ref{p wong 2.2 refined} is needed for the non-unital version of Theorem \ref{t 2-local $^*$-homomorphisms}}
\end{remark}

\begin{problem}\label{problem linearity} Is every (not necessarily linear) 2-local ($^*$-)homomorphism between C$^*$-algebras a ($^*$-)homomorphism? Equivalently, determine whether the hypothesis concerning linearity in Theorem \ref{t 2-local $^*$-homomorphisms} is superfluous.
\end{problem}

As we have commented before, we cannot expect that a local homomorphism between C$^*$-algebras is a homomorphism (see \cite[\S 2]{Pop}). We shall take a closer look at the counter-example provided by F. Pop.

\begin{example}\label{example Pop revisited} We know, from \cite[\S 2]{Pop}, that the mapping $T : M_2 (\mathbb{C}) \to M_4 (\mathbb{C})$, $$ T\left( \left(
                                                                      \begin{array}{cc}
                                                                        a & b \\
                                                                        c & d \\
                                                                      \end{array}
                                                                    \right)
 \right)  = \left(
              \begin{array}{cccc}
                a & 0 & b & 0 \\
                0 & a & 0 & c \\
                c & 0 & d & 0 \\
                0 & b & 0 & d \\
              \end{array}
            \right),
 $$ is a local homomorphism which is not multiplicative. Is easy to check that the above $T$ is a unital Jordan $^*$-homomorphism. We claim that $T$ is not a local $^*$-homomorphism. Otherwise, there exits a $^*$-homomorphism $$\pi =\pi_{\tiny\left(
                                                                                                   \begin{array}{cc}
                                                                                                     1 & 1 \\
                                                                                                     2 & 0 \\
                                                                                                   \end{array}
                                                                                                 \right)
 }: M_2 (\mathbb{C}) \to  M_4 (\mathbb{C})$$ satisfying $$\pi \left(
                                                                                                   \begin{array}{cc}
                                                                                                     1 & 1 \\
                                                                                                     2 & 0 \\
                                                                                                   \end{array}
                                                                                                 \right) = T\left(
                                                                                                   \begin{array}{cc}
                                                                                                     1 & 1 \\
                                                                                                     2 & 0 \\
                                                                                                   \end{array}
                                                                                                 \right) = \left(
              \begin{array}{cccc}
                1 & 0 & 1 & 0 \\
                0 & 1 & 0 & 2 \\
                2 & 0 & 0 & 0 \\
                0 & 1 & 0 & 0 \\
              \end{array}
            \right).$$ Therefore $$\pi \left(
                                                                                                   \begin{array}{cc}
                                                                                                     2 & 2 \\
                                                                                                     2 & 4 \\
                                                                                                   \end{array}
                                                                                                 \right) = \pi \left( \left(
                                                                                                   \begin{array}{cc}
                                                                                                     1 & 1 \\
                                                                                                     2 & 0 \\
                                                                                                   \end{array}
                                                                                                 \right) \left(
                                                                                                   \begin{array}{cc}
                                                                                                     1 & 1 \\
                                                                                                     2 & 0 \\
                                                                                                   \end{array}
                                                                                                 \right)^* \right)$$ $$ = \pi \left(
                                                                                                   \begin{array}{cc}
                                                                                                     1 & 1 \\
                                                                                                     2 & 0 \\
                                                                                                   \end{array}
                                                                                                 \right) \pi \left(
                                                                                                   \begin{array}{cc}
                                                                                                     1 & 1 \\
                                                                                                     2 & 0 \\
                                                                                                   \end{array}
                                                                                                 \right)^* =  \left(
              \begin{array}{cccc}
                2 & 0 & 2 & 0 \\
                0 & 5 & 0 & 1 \\
                2 & 0 & 4 & 0 \\
                0 & 1 & 0 & 1 \\
              \end{array}
            \right), $$
            $$\pi \left(
                                                                                                   \begin{array}{cc}
                                                                                                     5 & 1 \\
                                                                                                     1 & 1 \\
                                                                                                   \end{array}
                                                                                                 \right) = \pi \left( \left(
                                                                                                   \begin{array}{cc}
                                                                                                     1 & 1 \\
                                                                                                     2 & 0 \\
                                                                                                   \end{array}
                                                                                                 \right)^* \left(
                                                                                                   \begin{array}{cc}
                                                                                                     1 & 1 \\
                                                                                                     2 & 0 \\
                                                                                                   \end{array}
                                                                                                 \right) \right) =   \left(
              \begin{array}{cccc}
                5 & 0 & 1 & 0 \\
                0 & 2 & 0 & 2 \\
                1 & 0 & 1 & 0 \\
                0 & 2 & 0 & 4 \\
              \end{array}
            \right), $$ $$\pi \left(
                                                                                                   \begin{array}{cc}
                                                                                                     0 & 0 \\
                                                                                                     -2 & 4 \\
                                                                                                   \end{array}
                                                                                                 \right) = \left(
              \begin{array}{cccc}
                0 & 0 & 0 & 0 \\
                0 & 3 & 0 & -3 \\
                -2 & 0 & 4 & 0 \\
                0 & -1 & 0 & 1 \\
              \end{array}
            \right),$$ $$ \pi \left(
                                                                                                   \begin{array}{cc}
                                                                                                     -8 & 0 \\
                                                                                                     0 & 2 \\
                                                                                                   \end{array}
                                                                                                 \right) =\pi \left(\left(
                                                                                                   \begin{array}{cc}
                                                                                                     2 & 2 \\
                                                                                                     2 & 4 \\
                                                                                                   \end{array}
                                                                                                 \right) -2 \left(
                                                                                                   \begin{array}{cc}
                                                                                                     5 & 1 \\
                                                                                                     1 & 1 \\
                                                                                                   \end{array}
                                                                                                 \right) \right)= \left(
              \begin{array}{cccc}
                -8 & 0 & 0 & 0 \\
                0 & 1 & 0 & -3 \\
                0 & 0 & 2 & 0 \\
                0 & -3 & 0 & -7 \\
              \end{array}
            \right), $$ $$\pi \left(
                                                                                                   \begin{array}{cc}
                                                                                                     0 & 0 \\
                                                                                                     0 & 20 \\
                                                                                                   \end{array}
                                                                                                 \right) = \pi \left( \left(
                                                                                                   \begin{array}{cc}
                                                                                                     0 & 0 \\
                                                                                                     -2 & 4 \\
                                                                                                   \end{array}
                                                                                                 \right) \left(
                                                                                                   \begin{array}{cc}
                                                                                                     0 & 0 \\
                                                                                                     -2 & 4 \\
                                                                                                   \end{array}
                                                                                                 \right)^*\right)= 20 \left(
              \begin{array}{cccc}
                0 & 0 & 0 & 0 \\
                0 & 9/10 & 0 & -3/10 \\
                0 & 0 & 1 & 0 \\
                0 & -3/10 & 0 & 1/10 \\
              \end{array}
            \right),$$ $$\pi \left(
                                                                                                   \begin{array}{cc}
                                                                                                     1 & 0 \\
                                                                                                     0 & 0 \\
                                                                                                   \end{array}
                                                                                                 \right) = \left(
              \begin{array}{cccc}
                1 & 0 & 0 & 0 \\
                0 & 1/10 & 0 & 3/10 \\
                0 & 0 & 0 & 0 \\
                0 & 3/10 & 0 & 9/10 \\
              \end{array}
            \right),$$ $$ \pi \left(
                                                                                                   \begin{array}{cc}
                                                                                                     0 & 0 \\
                                                                                                     1 & 0 \\
                                                                                                   \end{array}
                                                                                                 \right) = \left(
              \begin{array}{cccc}
                0 & 0 & 0 & 0 \\
                0 & 3/10 & 0 & 9/10 \\
                1 & 0 & 0 & 0 \\
                0 & -1/10 & 0 & -3/10 \\
              \end{array}
            \right),$$ and $$\pi \left(
                                                                                                   \begin{array}{cc}
                                                                                                     0 & 1 \\
                                                                                                     0 & 0 \\
                                                                                                   \end{array}
                                                                                                 \right) = \left(
              \begin{array}{cccc}
                0 & 0 & 1 & 0 \\
                0 & 3/10 & 0 & -1/10 \\
                0 & 0 & 0 & 0 \\
                0 & 9/10 & 0 & -3/10 \\
              \end{array}
            \right),$$ which gives $\pi \left( \left(
                                                                                                   \begin{array}{cc}
                                                                                                     1 & 1 \\
                                                                                                     0 & 1 \\
                                                                                                   \end{array}
                                                                                                 \right) \left(
                                                                                                   \begin{array}{cc}
                                                                                                     1 & 1 \\
                                                                                                     0 & 0 \\
                                                                                                   \end{array}
                                                                                                 \right) \right) \neq \pi \left( \left(
                                                                                                   \begin{array}{cc}
                                                                                                     1 & 1 \\
                                                                                                     0 & 1 \\
                                                                                                   \end{array}
                                                                                                 \right) \right) \pi\left( \left(
                                                                                                   \begin{array}{cc}
                                                                                                     1 & 1 \\
                                                                                                     0 & 0 \\
                                                                                                   \end{array}
                                                                                                 \right) \right),$ contradicting that $\pi$ is a homomorphism.\smallskip

Furthermore, the elements $\displaystyle a= \left(
                                                                                                   \begin{array}{cc}
                                                                                                     1 & -1 \\
                                                                                                     1 & -1 \\
                                                                                                   \end{array}
                                                                                                 \right)$ and $\displaystyle b= \left(
                                                                                                   \begin{array}{cc}
                                                                                                     1 & 2 \\
                                                                                                     1 & 2 \\
                                                                                                   \end{array}
                                                                                                 \right)$ satisfy $a b=0$ and $T(a) T(b)\neq 0,$ which shows that $T$ does not preserves zero products. \end{example}

It seems natural to ask whether every local $^*$-homomorphism between C$^*$-algebras is multiplicative. We shall see that the answer to this question is, in general, negative. The next example illustrates this fact and provides an easier argument to Pop's counterexample.

\begin{example}\label{example unitarily equivalent matrices} A problem posed by P.R. Halmos in \cite[Proposition 159]{Hal54} asks  whether every square complex matrix is unitarily equivalent to its transpose. In other words, given $a\in M_{n} (\mathbb{C})$, when does there exist a unitary matrix $u\in M_n (\mathbb{C})$ satisfying  $u^* a u = a^{t}$?\smallskip

More generally, the problem of deciding whether two given square matrices $a$ and $b$ over the field of complex numbers are unitarily equivalent was positively solved by W. Specht \cite{Specht} who found a (more or less satisfactory) necessary and sufficient condition for two complex square matrices to be unitarily equivalent. In the setting of $2\times 2$ matrices the conditions are much more simple; F.D. Murnaghan \cite{Mur54} showed that, the traces of $a$, $a^2$, and $a^*a$ form a complete set of invariants to determine when two matrices in $M_2 (\mathbb{C})$ are unitarily equivalent {\rm(}i.e. two matrices $a,b\in M_2 (\mathbb{C})$ are unitarily equivalent if and only if $\hbox{tr}(a) = \hbox{tr}(b)$, $\hbox{tr}(a^2)=\hbox{tr}(b^2)$, and $\hbox{tr}(a^*a)= \hbox{tr}(b^*b)${\rm)}. Some years later, C. Pearcy \cite{Pearcy1962} obtained a list of nine conditions to determine when $a,b \in M_3(\mathbb{C})$ are unitarily equivalent {\rm(}see \cite{GarTener2012} for a recent publication on these topics{\rm )}.\smallskip

Murnaghan's characterization implies that every matrix in $M_2 (\mathbb{C})$ is unitarily equivalent to its transpose, that is for each $a\in M_2 (\mathbb{C})$ there exists a unitary matrix $u\in M_2 (\mathbb{C})$ {\rm(}depending on $a${\rm)} satisfying $u^* a u = a^t$. Consequently, the mapping $$T:M_2 (\mathbb{C}) \to M_2 (\mathbb{C}), \ \ T(a) = a^t,$$ is a linear local $^*$-homomorphism and a $^*$-anti-homomorphism which is not multiplicative.\end{example}

\textbf{Acknowledgments:} I would like to thank my friend and colleague Professor Timur Oikhberg for bringing reference \cite{Pop} to my attention, and for all the fruitful comments and suggestions provided by him. I would also like to thank the anonymous referee for the careful review and the helpful suggestions.

\end{document}